\documentclass[12pt]{amsart}
\usepackage{amssymb, amstext, amscd, amsmath}
\usepackage{fullpage}

\theoremstyle{plain}
\newtheorem{thm}{Theorem}[section]
\newtheorem{cor}[thm]{Corollary}
\newtheorem{prop}[thm]{Proposition}
\newtheorem{lem}[thm]{Lemma}

\theoremstyle{definition}
\newtheorem{defn}[thm]{Definition}
\newtheorem{eg}[thm]{Example}
\newtheorem{ques}[thm]{Question}

\theoremstyle{remark}
\newtheorem{rem}[thm]{Remark}

\begin{document}

\title{ spectrums of equivalent Schauder operators}
\title{ spectrums of equivalent Schauder operators}
\author[L. Y. Shi]{Luo Yi Shi}

\address{Department of Mathematics, Tianjin Polytechnic University, Tianjin, 300160,CHINA}

\email{sluoyi@yahoo.cn}

\author[Y. Cao]{Yang Cao}

\address{Institution of Mathematics, Jilin University, Changchun, 130012, China}

\email{caoyang@jlu.edu.cn}

\author[G. Tian]{Geng Tian}

\address{Institution of Mathematics, Jilin University, Changchun, 130012, China}

\email{tiangeng09@mails.jlu.edu.cn}



\begin{abstract}
 Assume that $T_1,T_2$ are equivalent Schauder operators. In this paper,
 we show that even in this case
 their Schauder spectrum may be very different in the view of operator theory. In fact, we get that if a
 self-adjoint Schauder
 operator $A$ has more than one points in its essential spectrum
$\sigma_e(A)$, then there exists a unitary spread operator $U$ such
that the Schauder spectrum $\sigma_S(UA)$ contains a ring which is
depended by the essential spectrum; if there is only one point in
$\sigma_e(A)$ and satisfies some conditions then there exists a
unitary spread operator $U$ such that the Schauder spectrum
$\sigma_S(UA)$ contains the circumference which is depended by the
essential spectrum.

\end{abstract}

\maketitle

\section{Introduction}

In their paper \cite{Cao}, Cao give an operator theory description
of bases on a separable Hilbert space $\mathcal{H}$. To study
operators on $\mathcal{H}$ from a basis theory viewpoint, it is
naturel to consider the behavior of operators related by equivalent
bases. For examples, they show that there always be some strongly
irreducible operators in the orbit of equivalent Schauder
matrices(\cite{Cao Yang 2}). However, in the usual way a spectral
method consideration of operators in the equivalent orbit is also
important to the joint research both on operator theory and Schauder
bases. Cao introduces the conception {\it Schauder spectrum} to do
this work. The main purpose of this paper is to show that the
Schauder spectrum of Schauder operators in a given orbit can be very
different.

Recall that a sequence of vectors $\{f_{n}\}_{n=1}^{\infty}$ in
$\mathcal{H}$ is said to be a {\it Schauder basis}
\cite{Singer,McArthur} for $\mathcal{H}$ if every element
$f\in\mathcal{H}$ has a unique series expansion $f=\sum c_nf_n$
which converges in the norm of $\mathcal{H}$. If $\{f_n\}$ is
Schauder basic for $\mathcal{H}$, the {\it sequence space associated
with} $\{f_n\}$ is defined to be the linear space of all sequences
$\{c_n\}$ for which $f=\sum c_nf_n$ is convergent. Two Schauder
bases $\{f_n\}_{n=1}^{\infty}$ and $\{g_n\}_{n=1}^{\infty}$ are {\it
equivalent} to each other if they have the same sequence space(cf,
\cite{Singer}, definition 12.1, p131, \cite{Garling}, p163). Denote
by $\omega$ the countable infinite cardinal. In paper \cite{Cao
Yang}, Cao.e.t considered the $\omega\times\omega$ matrix whose
column vectors comprise a Schauder basis and call them the {\it
Schauder matrix}. An operator has a Schauder matrix representation
under some ONB is called a {\it Schauder operator}. Given an
orthonormal basis(ONB in short) $\varphi=\{e_n\}_{n=1}^{\infty}$,
the vector $f_{n}$ in a Schauder basis sequence
$\psi=\{f_n\}_{n=1}^{\infty}$ corresponds an $l^{2}$ sequence
$\{f_{mn}\}_{m=1}^{\infty}$ defined uniquely by the series
$f_{n}=\sum_{m=1}^{\infty} f_{mn}e_{m}$. The matrix
$F_{\psi}=(f_{mn})_{\omega \times \omega}$ is called the {\it
Schauder matrix} of basis $\psi$ under the ONB $\varphi$.

Assume that $\psi_{1}, \psi_{2}$ are equivalent Schauder bases and $T_{\psi_{1}}, T_{\psi_{2}}$ are the
operators defined by Schauder matrices $F_{\psi_{1}}$ and $F_{\psi_{2}}$ respectively
under the same ONB.
Then there are no difference between  $\psi_1$
and $\psi_2$ from the view of bases of the Hilbert space. Are there
some notable differences between the operators $T_{\psi_{1}}$ and $T_{\psi_{2}}$ from
the view of operator theory? From the Arsove's theorem(\cite{Arsove}, or theorem 2.12 in \cite{Cao Yang}), there is some invertible
operator $X\in L(\mathcal{H})$ such that $XT_{\psi_{1}}= T_{\psi_{2}}$ holds. Hence
for a Schauder basis $\psi=\{f_{n}\}_{n=1}^{\infty}$, the set defined as
$$
\mathcal{O}_{gl}(\psi)=\{X\psi; X \in gl(\mathcal{H})\}
$$
in which $X\psi=\{Xf_{n}\}_{n=1}^{\infty}$ and $gl(\mathcal{H})$ consists of all invertible operators in $L(\mathcal{H})$
contains exactly all equivalent bases to $\psi$. Moreover, the set
$$
\mathcal{O}_{gl}(F_{\psi})=\{M_{X}F_{\psi}; M_{X} \hbox{ is the matrix of some operator }X \in gl(\mathcal{H})\}
$$
consists of all Schauder matrix equivalent to $F_{\psi}$. In the operator level, we define
$$
\mathcal{O}_{gl}(T_{\psi})=\{XT_{\psi}; X \in gl(\mathcal{H})\}.
$$
Then the set $\mathcal{O}_{gl}(T_{\psi})$ consists of operators related to bases equivalent to $\psi$.
Similarly, we consider following sets:
$$
\begin{array}{c}
\mathcal{O}_{u}(\psi)=\{U\psi; U \in U(\mathcal{H})\}, \\
\mathcal{O}_{u}(F_{\psi})=\{M_{U}F_{\psi}; M_{U} \hbox{ is the matrix of some unitary operator }U\}, \\
\mathcal{O}_{u}(T_{\psi})=\{UT_{\psi}; U \in U(\mathcal{H})\},
\end{array}
$$
where $U(\mathcal{H})$ consists of all unitary operators in
$L(\mathcal{H})$. Roughly speaking, by these set we bind operators
related to equivalent bases of the basis $\psi$ with the same basis
const. It is easy to check that a Schauder operator $T_{\psi}$ must be injective and having
a dense range. Denote by $T_{\psi} = UA_{\psi}$  the
polar decomposition of $T_{\psi}$, then the partial isometry $U$
must be a unitary operator. Then the orbit $\mathcal{O}_{u}(T_{\psi})$ is just the orbit $\mathcal{O}_{u}(A_{\psi})$
in which $A_{\psi}$ is the self-adjoint operator defined by the polar decomposition of $T_{\psi}$. In this paper
we focus on unitary operators with a nice basis theory understanding, that is, a slight generalization of
\textsl{spread form} defined by  W. T. Gowers and B. Maurey(\cite{Gowers 1}, \cite{Gowers 2}).

For a complex number $\lambda, \lambda$ will be called in the {\it
Schauder spectrum} of $T$ denoted by $\sigma_S(T)$ if and only if
there is no ONB such that $\lambda I-T$ has a matrix representation
as a Schauder matrix. It is obviously,
$\sigma(T)\supset\sigma_S(T)=\sigma_p(T)\cup\sigma_r(T)$ in which
$\sigma_r(T)=\{\lambda\in\mathbb{C}, \overline{\textup{Ran}(\lambda
I-T)}\neq\mathcal{H}\}$.

Now we state our main theorem:
\begin{thm}
Assume that $A$ is a self-adjoint Schauder operator.

\textup{(i)} If $\sigma(A)\subseteq[\lambda_1, \lambda_2],
\lambda_1>0$ and $\lambda_1,\lambda_2\in \sigma_e(A)$, then there
exists a unitary spread operator $U$ such that the Schauder spectrum
 $\sigma_S(UA)\supseteq R$ for any rings $R$ in the ring $R_{\lambda_1,
\lambda_2}^o$;

\textup{(ii)} If $\lambda_1,\lambda_2\in \sigma_e(A)$ and
$0<\lambda_1<\lambda_2$, then there exists a unitary spread operator
$U$ such that the Schauder spectrum $\sigma_S(UA)\supseteq R$ for
any rings $R$ in the ring $R_{\lambda_1, \lambda_2}^o$;

\textup{(iii)} If there exists only one point $\lambda_1\in
\sigma_e(A)$,  $\{t_k\}$ and $\{r_k\}$ contained in $\sigma(A)$
 and satisfy that $t_k <t_{k+1}, r_k > r_{k+1}$, $t_k\rightarrow
\lambda_1, r_k\rightarrow \lambda_1$, and
$\sum_{n=1}^\infty\prod_{k=1}^n(\frac{t_k}{\lambda_1})^2<\infty$,
$\sum_{n=1}^\infty\prod_{k=1}^n(\frac{\lambda_1}{r_k})^2<\infty$.
Then there exists a unitary spread operator $U$ such that  the
Schauder spectrum $\sigma_S(UA)\supseteq\{\lambda,
|\lambda|=\lambda_1\}$.

\end{thm}

That is, if $T$ is a Schauder operator, then there exist operator
$T_1 \in \mathcal{O}_{u}(T)$ such that $\sigma_S(T_1)$ has a certain thickness. Related concept will be
clear in later section.

We organize our paper as follows. In section 2, we introduce some
notations and  lemmas which will be used in the main theorem; in
section 3, we research the case that the spectrum of self-adjoint
Schauder operator has only two points; In section 4 we research the
case that the essential spectrum of self-adjoint Schauder operator
has only two points; In section 5, we research the case that there
is no point spectrum in the spectrum of self-adjoint Schauder
operator. At last, we get that if $A$ is a self-adjoint Schauder
operator with at least two essential spectrum, then exists $UA\in
\mathcal{O}_{u}(A)$ such that $\sigma_S(A)$ is thin and
$\sigma_S(UA)$ has a certain thickness.

\begin{rem}
In the seminar held at Jilin university, Cao shows that for a Schauder
operator $T$ there must be some unitary spread $U$ such that the Schauder operator
$UT$ has an empty Schauder spectrum. In this sense, our result in this paper
show that the Schauder spectrum of $UT$ may be very bad.
\end{rem}


\section{Notation and auxiliary results}

In this section we will introduce some notation  for convenience,
and some lemmas which will be used in the main theorem.

Throughout this paper, let $R_{\lambda_1, \lambda_2}=\{\lambda,
\lambda_1\leq|\lambda|\leq\lambda_2\}$, $R_{\lambda_1}=\{\lambda,
|\lambda|=\lambda_1\}$, $R_{\lambda_2}=\{\lambda,
|\lambda|=\lambda_2\}$ and $R_{\lambda_1, \lambda_2}^o=\{\lambda,
\lambda_1<|\lambda|<\lambda_2\}$ for $0<\lambda_1<\lambda_2$. If $E$
is a subset of complex plane $\mathbb{C}$ and $0\notin E$, let
$E^{-1}=\{\lambda, \frac{1}{\lambda}\in \textup{E}\}$,
$\textup{Card}\{E\}$ denote the cardinal number of $E$.

Recall the definition of the {\it spread from $A$ to $B$} given by
W. T. Gowers and B. Maurey.

\begin{defn}(\cite{Gowers 2}, p549)
Given an ONB $\{e_n\}_{n=1}^\infty$ and two infinite subsets $A,B$
of $\mathbb{N}$. Let $c_{00}$ be the vector space of all sequences
of finite support. Let the elements of $A$ and $B$ be written in
increasing order respectively as $\{a_1, a_2, \cdots \}$ and $\{b_1,
b_2, \cdots \}$. Then $e_n$ maps to $0$ if $n \notin A$, and
$e_{a_k}$ maps to $e_{b_k}$ for every $k \in N$. Denote this map by
$S_{A,B}$ and call it the spread from $A$ to $B$.
\end{defn}

Using spread forms, we can write some unitary operator into their
linear combination. See the Example 4.13 in \cite {Cao Yang}.

\begin{defn}(\cite{Cao Yang}, Definition 4.14)
A unitary operator $U$ on $\mathcal{H}$ is said to be a unitary
spread if there is a sequence $\{S_{A_n,B_n}\}_{n=1}^\infty$ of
spreads such that the series $\sum_{n=1}^\infty S_{A_n,B_n}$
converges to $U$ in strongly operator topology (SOT). Moreover, $U$
will be called a finite unitary spread if $U$ can be written as a
finite linear combination.
\end{defn}

In the paper \cite{Cao}, Cao.e.t proved that for each bijection
$\sigma$ on the set $\mathbb{N}$, the unitary operator $U_{\sigma}$
is a unitary spread.

\begin{lem}\label{bj}
Assume that $A$ is a self-adjoint operator satisfying that
$\sigma(A)\subseteq[\lambda_1, \lambda_2]$, $\lambda_1>0$, and there
exists $x\neq 0$ such that $||Ax||=\lambda_i||x||$. Then
$\lambda_i\in\sigma_p(A)$, and $x\in\textup{Ker}(\lambda_iI-A)$,
$i=1,2$.
\end{lem}

\begin{proof}
Indeed we only need to prove the case of $i=1$. The proof of the
case of $i=2$, is minor modifications of the proof of the analogous
statements in the case of $i=1$ by consider $A^{-1}$ and will be
omitted.

Since $\sigma(A)\subseteq[\lambda_1, \lambda_2]$, we know that $(Ax,
x)\geq\lambda_1||x||$ for any $x\neq 0$. Hence,
$||(\lambda_1I-A)x||^2=\lambda_1^2||x||+||Ax||^2-2\lambda_1(Ax,
x)\leq 0$, it follows that $||(\lambda_1I-A)x||=0$. That is to say
$x\in\textup{Ker}(\lambda_1I-A)$ and $\lambda_1\in\sigma_p(A)$.
\end{proof}

\begin{lem}\label{propbj}
Assume that $A$ is a self-adjoint operator satisfying that
$\sigma(A)\subseteq[\lambda_1, \lambda_2]$, $\lambda_1>0$. Then,
 for any unitary operator $U$,

\textup{(i)} $\sigma(UA)\subseteq R_{\lambda_1, \lambda_2}$;

\textup{(ii)} If $\lambda_1,\lambda_2\notin\sigma_p(A)$, then
$\sigma_p(UA)\cap R_{\lambda_i}=\emptyset$;if
$\lambda_1,\lambda_2\in\sigma_p(A)$, then
$\textup{Card}\{\sigma_p(UA)\cap R_{\lambda_i}\}\leq\textup{dim
Ker}(\lambda_iI-A), i=1,2$.
\end{lem}

\begin{proof}
(i) It is well known that if $T$ is an invertible operator, then
$\sigma(T^{-1})=\{\lambda, \lambda^{-1}\in \sigma(T)\}$, and
$r(T)\leq ||T||$ for any $T\in B(\mathcal{H})$. Thus
$\sigma(UA)=\frac{1}{\sigma((UA)^{-1})}$ and
$||UA||=||A||=\lambda_2,
||(UA)^{-1}||=||A^{-1}||=\frac{1}{\lambda_1}$. It follows that
$\lambda\notin \sigma(UA)$ when $|\lambda|>\lambda_2$ and
$\lambda\notin \sigma((UA)^{-1})$ when
$|\lambda|>\frac{1}{\lambda_1}$. Hence, $\sigma(UA)\subseteq
R_{\lambda_1, \lambda_2}$, for any unitary operator $U$.

(ii)Indeed we only need to prove the case of $i=1$. The proof of the
case of $i=2$, is minor modifications of the proof of the analogous
statements in the case of $i=1$ by consider $A^{-1}$ and will be
omitted.

 Assume $U$ is a unitary operator and
$\lambda\in\sigma_p(UA)\cap R_{\lambda_1}$. Then there exists $x\neq
0$ such that $UAx=\lambda x$ and $||Ax||=||UAx||=\lambda_1||x||$. By
Lemma \ref{bj}, $\lambda_1\in\sigma_p(A)$ and
$x\in\textup{Ker}(\lambda_1I-A)$. Hence, $A$ and $U$ have the matrix
forms
$$A= \left[\begin{array}{cc}
\lambda_1I&\\
&A_1
\end{array}\right
]
\begin{matrix}
\mbox{$\textup{Ker}(\lambda I-UA)$}\\
\mbox{$\textup{Ker}(\lambda I-UA)^{\perp}$}\\
\end{matrix}
,
U= \left[\begin{array}{cc}
U_{11}&U_{12}\\
U_{21}&U_{22}
\end{array}\right
]
\begin{matrix}
\mbox{$\textup{Ker}(\lambda I-UA)$}\\
\mbox{$\textup{Ker}(\lambda I-UA)^{\perp}$}\\
\end{matrix},$$
and $\textup{Ker}(\lambda I-UA)\subseteq
\textup{Ker}(\lambda_1I-A)$.

For any $x\in\textup{Ker}(\lambda I-UA)$, $UAx=\lambda x$. Since $U$
is a  unitary operator, it is easy to check that $U_{12}=U_{21}=0$.
Hence, $UA$ has the matrix form
$$UA= \left[\begin{array}{cc}
\lambda_1U_{11}&\\
&U_{22}A_1
\end{array}\right
]
\begin{matrix}
\mbox{$\textup{Ker}(\lambda I-UA)$}\\
\mbox{$\textup{Ker}(\lambda I-UA)^{\perp}$}\\
\end{matrix}
,$$ in which $U_{11}$ and $U_{22}$ are unitary operators and
$$\sigma_P(UA)=\sigma_P(\lambda_1U_{11})\cup\sigma_P(U_{22}A_1),$$
$$\textup{Card}\{\sigma_P(\lambda_1U_{11})\}\leq\textup{dim
Ker}(\lambda I-UA)\leq\textup{dim Ker}(\lambda_1I-A).$$

If there exists another $\delta\in\sigma_p(U_{22}A_1)\cap
R_{\lambda_1}$, repeating the above process, we can get that
$\textup{Ker}(\delta I-U_{22}A_1)\subseteq
\textup{Ker}(\lambda_1I-A)$ and $\textup{Ker}(\delta
I-U_{22}A_1)\bot \textup{Ker}(\lambda I-UA)$.

Repeating the above process, we can obtain that
$\textup{Card}\{\sigma_p(UA)\cap R_{\lambda_1}\}\leq\textup{dim
Ker}(\lambda_1I-A).$

\end{proof}

\begin{rem}
By the above lemma, we know that if the spectrum $\sigma(A)$ of a
self-adjoint Schauder operator is contained in an interval, then the
Schauder spectrum $\sigma_S(UA)$ must be contained in the ring which
is depended by the interval.
\end{rem}

\section{Only two points in $\sigma(A)$}

In this section, we will research the case that the spectrum of
self-adjoint Schauder operator $A$ has only two points $\lambda_1,
\lambda_2$ and $0<\lambda_1<\lambda_2$.

According to Lemma \ref{propbj}, we know that for any unitary
operator $U$, there exists at most denumerable subsets $\sigma_1$ in
$R_{\lambda_1}$ and $\sigma_2$ in $R_{\lambda_2}$ such that
$\sigma_p(UA)\subseteq\sigma_1\cup\sigma_2\cup R_{\lambda_1,
\lambda_2}^o$. In this section, we will show that if
ker$(\lambda_i-A)=\infty, i=1,2$, then for
 any at most denumerable subsets $\sigma_1$ in
$R_{\lambda_1}$, $\sigma_2$ in $R_{\lambda_2}$ and a ring $R$ in
$R_{\lambda_1, \lambda_2}^o$, there exists a unity operator $U$ such
that $\sigma_p(UA)\subseteq\sigma_1\cup\sigma_2\cup R$. Hence, there
exists $UA\in \mathcal{O}_{u}(A)$ such that $\sigma_S(A)$ is thin
and $\sigma_S(UA)$ has a certain thickness.

\begin{lem}\label{lem1}
Assume that $A$ is a self-adjoint operator satisfying that
$\sigma(A)=\{\lambda_1,\lambda_2\}$, $0<\lambda_1<\lambda_2$ and dim
ker$(\lambda_i-A)=\infty$. Then, there exists a unitary spread
operator $U$ such that $\sigma_p(UA)=R_{\lambda_1, \lambda_2}^o$.
\end{lem}

\begin{proof}

By the classical spectral theory of normal operator, we have
following orthogonal decomposition of $A$,

$$A=\oplus_{n\in\mathbb{Z}}A_n,$$
in which $A_0=\lambda_2I, A_{-1}=\lambda_1I,$ $A_n=\lambda_1I$ for
all $n\geq 1$ and $A_n=\lambda_2I$ for all $n\leq -2$.

Now we choose an ONB $\{e_k^{(n)}\}_{k=1}^\infty$, for each $n\in
\mathbb{Z}$. And let $U$ be the unitary spread operator defined as

$$Ue_k^{(n)}=e_k^{(n+1)}, n\in\mathbb{Z}, k\in \mathbb{N}.$$

For a vector $x\in \mathcal{H}$ now under the ONB constructed it has
a $l_2$-sequence coordinate in the form

$$x=\sum_{n\in\mathbb{Z}}x^{(n)}=\sum_{n\in\mathbb{Z}}\sum_{k=1}^{\infty}x_k^{(n)}e_k^{(n)}.$$

Now simply we have

$$UAx=UA(\sum_{n\in\mathbb{Z}}x^{(n)})=\sum_{n\in\mathbb{Z}}UAx^{(n)}=\sum_{n\in\mathbb{Z}}A_{n-1}x^{(n-1)}.$$

Now suppose for some $\lambda\neq 0$ we do have some vector $x$ such
that $(\lambda I-UA)x=0$, then we have

$$\lambda x^{(n)}=A_{n-1}x^{(n-1)}.$$

Therefore, following equations hold:

$$x^{(n)}=\lambda^{-n}A_{n-1}A_{n-2}\cdots A_0 x^{(0)}, n\geq 1,$$

$$x^{(n)}=\lambda^{-n}A_{n}^{-1}A_{n+1}^{-1}\cdots A_{-1}^{-1}x^{(0)}, n\leq -1.$$

That is to say $x^{(n)}=\lambda^{-n}\lambda_1^{n-1}\lambda_2x^{(0)},
n\geq 1$ and
$x^{(n)}=\lambda^{-n}\lambda_1^{-1}\lambda_2^{n+1}x^{(0)},n\leq -1.$

Since $0<\lambda_1<\lambda_2$ it is easy to see that if
$\lambda_1<|\lambda|<\lambda_2$ then
$\sum_{n\in\mathbb{Z}}||x^{(n)}||^2<\infty$; if
$\lambda_1\leq|\lambda|$ then
 $||x^{(n)}||>1$ for $n\geq 1$, if
$|\lambda|\geq\lambda_2$ then $||x^{(-n)}||>1$ for $n\geq 1$, i.e.
if $\lambda_1\leq|\lambda|$ or $|\lambda|\geq\lambda_2$ then
$\sum_{n\in\mathbb{Z}}||x^{(n)}||^2=\infty$. Hence,
$\sigma_p(UA)=\{\lambda, \lambda_1<|\lambda|<\lambda_2\}$.

\end{proof}

\begin{prop}\label{lem2}
Assume that $A$ is a self-adjoint operator satisfying that
$\sigma(A)=\{\lambda_1,\lambda_2\}$, $0<\lambda_1<\lambda_2$ and dim
ker$(\lambda_i-A)=\infty$. Then there exists a unitary spread
operator $U$ such that $\sigma_p(UA)=R$ for any rings $R$ in the
ring $R_{\lambda_1, \lambda_2}^o$.
\end{prop}

\begin{proof}

Firstly, we prove that if $R=\{\lambda,
(\lambda_1^{n_1}\lambda_2^{n_2})^\frac{1}{n_1+n_2}<|\lambda|<(\lambda_1^{m_1}\lambda_2^{m_2})^\frac{1}{m_1+m_2}\}$
is a ring in $R_{\lambda_1, \lambda_2}^o$ for some  integers $n_1,
n_2, m_1, m_2$, then there exists a unitary spread operator $U$ such
that $\sigma_p(UA)=R$.

We assign the same notations used in the proof of Lemma \ref{lem1}.

$$x^{(n)}=\lambda^{-n}A_{n-1}A_{n-2}\cdots A_0 x^{(0)}, n\geq 1,$$

$$x^{(n)}=\lambda^{-n}A_{n}^{-1}A_{n+1}^{-1}\cdots A_{-1}^{-1}x^{(0)}, n\leq -1.$$

Let $$A_0=A_1=\cdots A_{n_1}=\lambda_1I,$$
$$A_{n_1+1}=A_{n_1+2}=\cdots A_{n_1+n_2}=\lambda_2I,$$
$$\vdots$$
$$A_{k_1n_1+k_2n_2+k_3}=\lambda_1I,$$ for any $k_1, k_2$ and $1\leq
k_3\leq n_1$, $$A_{k_1n_1+k_2n_2+k_3}=\lambda_2I,$$ for any $k_1,
k_2$ and $n_1+1\leq k_3\leq (n_1+n_2)$.

And let $$A_{-1}=\cdots A_{-m_2}=\lambda_2I,$$
$$A_{-m_2-1}=A_{-m_2-2}=\cdots A_{-m_2-m_1}=\lambda_1I,$$
$$\vdots$$
$$A_{-k_1m_2-k_2m_1-k_3}=\lambda_2I,$$ for any $k_1, k_2$ and $1\leq
k_3\leq m_1$, $$A_{-k_1m_2-k_2m_1-k_3}=\lambda_1I,$$ for any $k_1,
k_2$ and $m_2+1\leq k_3\leq (m_1+m_2)$.

Then we have
$$x^{(k(n_1+n_2))}=\lambda^{-k(n_1+n_2)}\lambda_1^{kn_1}\lambda_2^{kn_2}x^{(0)},
k\geq 1,$$ and
$$x^{-(k(m_1+m_2))}=\lambda^{-k(m_1+m_2)}\lambda_1^{-km_1}\lambda_2^{-km_2}x^{(0)},k\geq
1.$$

It is easy to see that if
$(\lambda_1^{n_1}\lambda_2^{n_2})^\frac{1}{n_1+n_2}<|\lambda|<(\lambda_1^{m_1}\lambda_2^{m_2})^\frac{1}{m_1+m_2}$
then $\sum_{n\in\mathbb{Z}}||x^{(n)}||^2<\infty$; if
$(\lambda_1^{n_1}\lambda_2^{n_2})^\frac{1}{n_1+n_2}\geq|\lambda|$
there exists $N$ such that $||x^{(n)}||>1$ for $n>N$, if
$|\lambda|\geq(\lambda_1^{m_1}\lambda_2^{m_2})^\frac{1}{m_1+m_2}$
there exists $N$ such that $||x^{(-n)}||>1$ for $n>N$, i.e. if
$\lambda_1\leq|\lambda|$ or $|\lambda|\geq\lambda_2$ then
$\sum_{n\in\mathbb{Z}}||x^{(n)}||^2=\infty$. Hence,
$\sigma_p(UA)=R$.

Now we turn to the more general situation.

Since
$\lim_{n_1\rightarrow\infty}(\lambda_1^{n_1}\lambda_2^{n_2})^\frac{1}{n_1+n_2}=\lambda_1,
\lim_{n_2\rightarrow\infty}(\lambda_1^{n_1}\lambda_2^{n_2})^\frac{1}{n_1+n_2}=\lambda_2
$. We can get that  there exists a unitary spread operator $U$ such
that $\sigma_p(UA)=R$ for any rings $R$ in the ring $R_{\lambda_1,
\lambda_2}^o$.

\end{proof}

\begin{lem}\label{lem3}
Assume that $A$ is a self-adjoint operator satisfying that
$\sigma(A)=\{\lambda_1,\lambda_2\}$, $\lambda_1<\lambda_2$ and dim
ker$(\lambda_i-A)=\infty$. Then, there exists a unitary spread
operator $U$ such that $\sigma_p(UA)=\sigma_1\cup\sigma_2$ for any
at most denumerable subsets $\sigma_1$ in $\{\lambda,
|\lambda|=\lambda_1\}$ and $\sigma_2$ in $\{\lambda,
|\lambda|=\lambda_2\}$.
\end{lem}

\begin{proof}
Since $A$ is a self-adjoint operator, by the classical spectral
theory of normal operator, we have following orthogonal
decomposition of $A$
\[\begin{matrix}\begin{bmatrix}
\lambda_1I&\\
&\lambda_2I
\end{bmatrix}&\end{matrix}.\]
Let $U=U_1\oplus U_2$, in which $U_1, U_2$ are  unity operators such
that  $\sigma_p(\lambda_1U_1)=\sigma_1,
\sigma_p(\lambda_2U_2)=\sigma_2$. Then $U$ is a unity operator and
$\sigma_p(UA)=\sigma_1\cup\sigma_2$.
\end{proof}

According to the Lemmas \ref{propbj}, \ref{lem1}, \ref{lem3} and the
Proposition \ref{lem2}, we can get the following theorem:

\begin{thm}\label{prop1}
Assume that $A$ is a self-adjoint operator satisfying that
$\sigma(A)=\{\lambda_1,\lambda_2\}$, $0<\lambda_1<\lambda_2$ and dim
ker$(\lambda_i-A)=\infty$. Then, there exists a unitary spread
operator $U$ such that $\sigma_p(UA)=\sigma_1\cup\sigma_2\cup R$ for
any at most denumerable subsets $\sigma_1$ in $\{\lambda,
|\lambda|=\lambda_1\}$, $\sigma_2$ in $\{\lambda,
|\lambda|=\lambda_2\}$ and $R$ is a ring in the ring $R_{\lambda_1,
\lambda_2}^o$. Moreover, for any unitary operator $U$, there exists
at most denumerable subsets $\sigma_1$ in $R_{\lambda_1}$ and
$\sigma_2$ in $R_{\lambda_2}$ such that
$\sigma_p(UA)\subseteq\sigma_1\cup\sigma_2\cup R_{\lambda_1,
\lambda_2}^o$.
\end{thm}

\begin{proof}
Since $A$ is a self-adjoint operator, by the classical spectral
theory of normal operator, we have following orthogonal
decomposition of $A=A_1\oplus A_1$ where $A_1$ is a self-adjoint
operator satisfying that $\sigma(A_1)=\{\lambda_1,\lambda_2\}$ and
dim ker$(\lambda_i-A)=\infty$. By Lemmas \ref{lem1}, \ref{lem3} and
the Proposition \ref{lem2}, we get that there exists a unitary
operator $U$ such that $\sigma_p(UA)=\sigma_1\cup\sigma_2\cup R$ for
any at most denumerable subsets $\sigma_1$ in $\{\lambda,
|\lambda|=\lambda_1\}$, $\sigma_2$ in $\{\lambda,
|\lambda|=\lambda_2\}$ and $R$ is a ring in the ring $R_{\lambda_1,
\lambda_2}^o$. The last part of this theorem is obvious by the Lemma
\ref{propbj}.
\end{proof}

\begin{rem}
By the above theorem, we know that if the spectrum $\sigma(A)$ of a
self-adjoint Schauder operator has only two points $\lambda_1,
\lambda_2$ and ker$(\lambda_i-A)=\infty, i=1,2$, then for any ring
$R$ in $R_{\lambda_1, \lambda_2}^o$ and at most denumerable subsets
$\sigma_1$ in $\{\lambda, |\lambda|=\lambda_1\}$, $\sigma_2$ in
$\{\lambda, |\lambda|=\lambda_2\}$, there exists $UA\in
\mathcal{O}_{u}(A)$ such that $\sigma_S(UA)$ contains
$\sigma_1\cup\sigma_2\cup R$. i.e. $\sigma_S(UA)$  has a certain
thickness, $\sigma_S(A)$ is thin. In other words, there is no ONB
such that $\lambda I-UA$ has a matrix representation as a Schauder
matrix for $\lambda\in \sigma_1\cup\sigma_2\cup R$.
\end{rem}

\section{Only two points in $\sigma_e(A)$}

In this section, we will research the case that the essential
spectrum of self-adjoint operator $A$ has only two points
$\lambda_1, \lambda_2$ and $0<\lambda_1<\lambda_2$. We will show
that for any rings $R$ in the ring $R_{\lambda_1,
\lambda_2}^o=\{\lambda, |\lambda_1|<|\lambda|<\lambda_2\}$, there
exists a unitary spread operator $U$ such that
$R_{\lambda_1\lambda_2}\supseteq\sigma_p(UA)\supseteq R$. i.e. there
exists $UA\in \mathcal{O}_{u}(A)$ such that $\sigma_S(A)$ is thin
and $\sigma_S(UA)$ has a certain thickness.

\begin{thm}\label{lem5}
Assume that $A$ is a self-adjoint operator satisfying the following
properties:

\textup{(i)} $\sigma(A)=\sigma_p(A)\cup \{\lambda_1, \lambda_2\}$,
$0<\lambda_1<\lambda_2$ and $\lambda_1, \lambda_2$ are the unique
accumulation points of $\sigma(A)$;

\textup{(ii)} For each $t\in \sigma_p(A)$, dim $\ker (A-tI)=1$.

Then there exists a unitary spread operator $U$ such that
$R_{\lambda_1\lambda_2}\supseteq\sigma_p(UA)\supseteq R$ for any
rings $R$ in the ring $R_{\lambda_1, \lambda_2}^o=\{\lambda,
|\lambda_1|<|\lambda|<\lambda_2\}$.

Moreover, if $t_k > t_{k+1}, r_k < r_{k+1}$ for all $k$, then there
exists a unitary spread operator $U$ such that $\sigma_p(UA)=R$ for
any rings $R$ in the ring $R_{\lambda_1, \lambda_2}^o=\{\lambda,
|\lambda_1|<|\lambda|<\lambda_2\}$; for any unitary operator $U$,
$\sigma_p(UA)\subset R_{\lambda_1, \lambda_2}=\{\lambda,
|\lambda_1|\leq|\lambda|\leq\lambda_2\}$ and
$\textup{Card}\{\sigma_p(UA)\cap
R_{\lambda_1}\}\leq1,\textup{Card}\{\sigma_p(UA)\cap
R_{\lambda_1}\}\leq1$.
\end{thm}

\begin{proof}
We only prove the case that $R=R_{\lambda_1\lambda_2}$, the proof of
the more general situation is similar to the Proposition \ref{lem2}
and we omit it.

The self-adjoint operator satisfying the conditions appearing in the
proposition has a spectrum in the following form:
$$\sigma(A)=\{t_1, t_2,\cdots, t_k,\cdots\}\cup\{r_1, r_2,\cdots, r_k,\cdots\}\cup\{\lambda_1, \lambda_2\},$$
in which $\lambda_1$ is the accumulation point of the sequence
$\{t_k\}$, $\lambda_2$ is the accumulation point of the sequence
$\{r_k\}$.

Choose the subsequences $\{t_{nk}\}_{k=1}^{\infty}$, $n\geq 0$ of
$\{t_k\}$ and $\{r_{nk}\}_{k=1}^{\infty}$, $n\geq 1$ of $\{r_k\}$
satisfying the following properties:

(i) $\lim_{k\rightarrow \infty}t_{nk}=\lim_{n\rightarrow
\infty}t_{nk}=\lambda_1$, $\lim_{k\rightarrow
\infty}r_{nk}=\lim_{n\rightarrow \infty}r_{nk}=\lambda_2$;

(ii) There exist $t_{nk}$ and $r_{nk}$ such that $t_{nk}=t_{k_0}$,
$r_{nk}=r_{k_1}$ for any $t_{k_0}\in\{t_k\},r_{k_1}\in\{r_k\}$;

(iii) $t_{n_1k_1}\neq t_{n_2k_2}, r_{n_1k_1}\neq r_{n_2k_2}$ when
$n_1\neq n_2$ or $k_1\neq k_2$.

Let $J_n=\{t_{nk}\}, n\geq 0, J_n=\{r_{-nk}\}, n\leq -1.$ We
rearrange these intervals as follows: $I_0=J_0, I_n=J_{n+1}$ for
$n\geq 1$, $I_{-1}=J_{1}, I_n=J_{n+1}$ for $n\leq -2$.

Denote $E_n=E_{I_n}$ the spectral projection on the interval $I_n$
and by $H_n=$Ran$(E_n)$ for $n\in \mathbb{Z}$. Now we choose an ONB
$\{e_k^{(n)}\}_{k=1}^\infty$, for each $n\in \mathbb{Z}$. Since each
$H_n$ is a reducing subspace of $A$, we can write $A$ into the
direct sum:
$$A=\oplus_{n=-\infty}^{+\infty}A_n.$$

Let $\sup_k\{t_{nk}\}=\alpha_n^{(1)},
\inf_k\{t_{nk}\}=\alpha_n^{(2)}$ for $n\geq 0$,
$\sup_k\{r_{nk}\}=\beta_n^{(1)}, \inf_k\{r_{nk}\}=\beta_n^{(2)}$ for
$n\geq 1$. Then $\lim_{n\rightarrow
\infty}\alpha_n^{(1)}=\lim_{n\rightarrow
\infty}\alpha_n^{(2)}=\lambda_1, \lim_{n\rightarrow
\infty}\beta_n^{(1)}=\lim_{n\rightarrow
\infty}\beta_n^{(2)}=\lambda_2$.

Now let $U$ be the unitary spread operator defined as

$$Ue_k^{(n)}=e_k^{(n+1)}, n\in\mathbb{Z}, k\in \mathbb{N}.$$

For a vector $x\in \mathcal{H}$ now under the ONB constructed it has
a $l_2$-sequence coordinate in the form

$$x=\sum_{n\in\mathbb{Z}}x^{(n)}=\sum_{n\in\mathbb{Z}}\sum_{k=1}^{\infty}x_k^{(n)}e_k^{(n)}.$$

Now simply we have

$$UAx=UA(\sum_{n\in\mathbb{Z}}x^{(n)})=\sum_{n\in\mathbb{Z}}UAx^{(n)}=\sum_{n\in\mathbb{Z}}A_{n-1}x^{(n-1)}.$$

Now suppose for some $\lambda\neq 0$ we do have some vector $x$ such
that $(\lambda I-UA)x=0$, then we have

$$\lambda x^{(n)}=A_{n-1}x^{(n-1)}.$$

Therefore, following equations hold:

$$x^{(n)}=\lambda^{-n}A_{n-1}A_{n-2}\cdots A_0 x^{(0)}, n\geq 1,$$

$$x^{(n)}=\lambda^{-n}A_{n}^{-1}A_{n+1}^{-1}\cdots A_{-1}^{-1}x^{(0)}, n\leq -1.$$

Hence,
$$\lambda^{-n}\beta_0^{(2)}\alpha_2^{(2)}\alpha_3^{(2)}\cdots\alpha_n^{(2)}\leq||x^{(n)}||
\leq\lambda^{-n}\beta_0^{(1)}\alpha_2^{(1)}\alpha_3^{(1)}\cdots\alpha_n^{(1)},
 n\geq 1;$$
 $$\frac{\lambda^{-n}}{ \alpha_1^{(1)}\beta_1^{(1)}\beta_2^{(1)}\cdots\beta_{-n-1}^{(1)}}
\leq||x^{(n)}||
\leq\frac{\lambda^{-n}}{\alpha_1^{(2)}\beta_1^{(2)}\beta_2^{(2)}\cdots\beta_{-n-1}^{(2)}},
n\leq -1 .$$

Since $\lim_{n\rightarrow \infty}\alpha_n^{(1)}=\lim_{n\rightarrow
\infty}\alpha_n^{(2)}=\lambda_1, \lim_{n\rightarrow
\infty}\beta_n^{(1)}=\lim_{n\rightarrow
\infty}\beta_n^{(2)}=\lambda_2$, it is easy to see that if
$\lambda_1<|\lambda|<\lambda_2$ then
$\sum_{n\in\mathbb{Z}}||x^{(n)}||^2<\infty$; if
$\lambda_1<|\lambda|$ there exists $N$ such that $||x^{(n)}||>1$ for
$n>N$, if $|\lambda|>\lambda_2$ there exists $N$ such that
$||x^{(-n)}||>1$ for $n>N$, i.e. if $\lambda_1<|\lambda|$ or
$|\lambda|>\lambda_2$ then
$\sum_{n\in\mathbb{Z}}||x^{(n)}||^2=\infty$. That is to say
$\{\lambda,
\lambda_1\leq|\lambda|\leq\lambda_2\}\supseteq\sigma_p(UA)\supseteq\{\lambda,
\lambda_1<|\lambda|<\lambda_2\}$.

Moreover,  if $t_k > t_{k+1}, r_k < r_{k+1}$ for all $k$ then
$t_k>\lambda_1, r_k<\lambda_2$. So $\alpha_n^{(i)}\geq\lambda_1,
\beta_n^{(i)}\leq\lambda_2$ for all $n$ and $i=1,2$.
 It is easy to see that
$\sigma_p(UA)=\{\lambda, |\lambda_1|<|\lambda|<\lambda_2\}$.
Furthermore, by Lemma \ref{propbj}, we get that for any unitary
operator $U$, $\sigma_p(UA)\subset \{\lambda,
|\lambda_1|\leq|\lambda|\leq\lambda_2\}$ and
$\textup{Card}\{\sigma_p(UA)\cap
R_{\lambda_1}\}\leq1,\textup{Card}\{\sigma_p(UA)\cap
R_{\lambda_1}\}\leq1$.
\end{proof}

\begin{rem}
\textup{(i)} Trivial modifications adapt the proof of Theorem
\ref{lem5}, we can weaken the condition dim $\ker (A-tI)=1$ to dim
$\ker (A-tI)<\infty$.

\textup{(ii)} In the Theorem \ref{lem5}, we obtained that there
exists a unitary operator $U$ such that $\sigma_p(UA)\supseteq R$
for any rings $R$ in the ring $R_{\lambda_1\lambda_2}$. Moreover, we
got $\sigma_p(UA)=R_{\lambda_1\lambda_2}$ if adding the condition
that $t_k > t_{k+1}, r_k < r_{k+1}$ for all $k$. The following
examples illustrate that this condition is necessary.
\end{rem}

\begin{eg}\label{eg1}
We assign the same notations used in the Theorem \ref{lem5}.

(1) Let $\lambda_1=1$, $\lambda_2>1$, and $t_{n1}=1-\frac{1}{n}$,
$t_{nk}=\frac{k+n-1}{k+n}+\frac{\frac{k+n}{k+n+1}-\frac{k+n-1}{k+n}}{k+n-1}\cdot
n$ for $n\geq 1, k\geq 2$ and $r_k < r_{k+1}$ for all $k\geq 1$.
Then according to the proof of Theorem \ref{lem5} and let
$A_n=\oplus_{k=1}^\infty t_{nk}$, $x^{(0)}=e_0^{(0)}$ in Theorem
\ref{lem5}, we obtain that $\sigma_p(UA)=\{\lambda,
|\lambda_1|\leq|\lambda|<\lambda_2\}$.

(2) Let $\lambda_2=1$, $\lambda_1<1$, and $r_{n1}=1+\frac{1}{n}$,
$r_{nk}=\frac{k+n+1}{k+n}-\frac{\frac{k+n+1}{k+n}-\frac{k+n+2}{k+n+1}}{k+n-1}\cdot
n$ for $n\geq 1, k\geq 2$ and $t_k > t_{k+1}$ for all $k\geq 1$.
Then according to the proof of Theorem \ref{lem5} and let
$A_n=\oplus_{k=1}^\infty t_{nk}$, $x^{(0)}=e_0^{(0)}$ in Theorem
\ref{lem5}, we obtain that $\sigma_p(UA)=\{\lambda,
|\lambda_1|<|\lambda|\leq\lambda_2\}$.

(3) Let $\lambda_1=1$, $\lambda_2=2$, and $t_{n1}=1-\frac{1}{n}$,
$r_{n1}=2+\frac{2}{n}$,
$t_{nk}=\frac{k+n-1}{k+n}+\frac{\frac{k+n}{k+n+1}-\frac{k+n-1}{k+n}}{k+n-1}\cdot
n$,
$r_{nk}=(2+\frac{2}{k+n-1})-\frac{(2+\frac{2}{k+n-1})-(2+\frac{2}{k+n})}{k+n-1}\cdot
n$ for $n\geq 1, k\geq 2$. Then according to the proof of Theorem
\ref{lem5} and let $A_n=\oplus_{k=1}^\infty t_{nk}$,
$x^{(0)}=e_0^{(0)}$ in Theorem \ref{lem5}, we obtain that
$\sigma_p(UA)=\{\lambda, |\lambda_1|\leq|\lambda|\leq\lambda_2\}$.
\end{eg}

 Trivial modifications adapt the proof of the Theorem \ref{lem5},
we can get the following Proposition.

\begin{cor}\label{lem6}
Assume that $A$ is a self-adjoint operator satisfying the following
properties:

\textup{(i)} $\sigma(A)=\sigma_p(A)\cup \{\lambda_1\}$,
$0<\lambda_1$ and $\lambda_1$ is the unique accumulation point of
$\sigma(A)$;

\textup{(ii)} For each $t\in \sigma_p(A)$, dim $\ker (A-tI)<\infty$;

\textup{(iii)} $\sigma_p(A)=\{t_1, t_2,\cdots,
t_k,\cdots\}\cup\{r_1,
r_2,\cdots, r_k,\cdots\}$, $t_k <t_{k+1}, r_k > r_{k+1}$, and\\
$\sum_{n=1}^\infty\prod_{k=1}^n(\frac{t_k}{\lambda_1})^2<\infty$,
$\sum_{n=1}^\infty\prod_{k=1}^n(\frac{\lambda_1}{r_k})^2<\infty$.

Then, there exists a unitary spread operator $U$ such that
$\sigma_p(UA)=\{\lambda, |\lambda|=\lambda_1\}$.
\end{cor}

\begin{eg}

Let $A$ is a self-adjoint operator satisfying that
$\sigma(A)=\sigma_p(A)\cup \{1\}$, $\sigma_p(A)=\{t_{nk},
r_{nk}\}_{k,n=1}^\infty$, in which $t_{n1}=1-\frac{1}{n}$,
$t_{nk}=\frac{k+n-1}{k+n}+\frac{\frac{k+n}{k+n+1}-\frac{k+n-1}{k+n}}{k+n-1}\cdot
n$, $r_{n1}=1+\frac{1}{n}$,
$r_{nk}=\frac{k+n+1}{k+n}-\frac{\frac{k+n+1}{k+n}-\frac{k+n+2}{k+n+1}}{k+n-1}\cdot
n$ for $n\geq 1, k\geq 2$, for each $t\in \sigma_p(A)$, dim $\ker
(A-tI)=1$. By Corollary \ref{lem6}, and (1), (2) of Example
\ref{eg1}, we can get that  there exists a unitary spread operator
$U$ such that $\sigma_p(UA)=\{\lambda, |\lambda|=1\}$.
\end{eg}

\begin{rem}
By the Theorem \ref{lem5}, we know that if the essential spectrum of
self-adjoint operator $A$ has only two points $\lambda_1, \lambda_2$
and $0<\lambda_1<\lambda_2$ and  for each $t\in \sigma_p(A)$, dim
$\ker (A-tI)<\infty$, then for any ring $R$ in $R_{\lambda_1,
\lambda_2}^o$, there exists $UA\in \mathcal{O}_{u}(A)$ such that
$\sigma_S(UA)$ contains $R$. i.e. $\sigma_S(UA)$  has a certain
thickness, $\sigma_S(A)$ is thin. In other words, there is no ONB
such that $\lambda I-UA$ has a matrix representation as a Schauder
matrix for $\lambda\in R$.
\end{rem}

\section{No points spectrum in $\sigma(A)$}

In this section, we will research the case that there is no point
spectrum in $\sigma(A)$. i.e. $\sigma(A)=[\lambda_1, \lambda_2],
0<\lambda_1$.

According to Lemma \ref{propbj}, we know that for any unitary
operator $U$, there exists at most denumerable subsets $\sigma_1$ in
$R_{\lambda_1}$ and $\sigma_2$ in $R_{\lambda_2}$ such that
$\sigma_p(UA)\subseteq\sigma_1\cup\sigma_2\cup R_{\lambda_1,
\lambda_2}^o$. In this section, we will show that if
ker$(\lambda_i-A)=\infty, i=1,2$, then for
 any at most denumerable subsets $\sigma_1$ in
$R_{\lambda_1}$, $\sigma_2$ in $R_{\lambda_2}$ and a ring $R$ in
$R_{\lambda_1, \lambda_2}^o$, there exists a unity operator $U$ such
that $\sigma_p(UA)\subseteq\sigma_1\cup\sigma_2\cup R$. i.e. there
exists $UA\in \mathcal{O}_{u}(A)$ such that $\sigma_S(A)$ is thin
and $\sigma_S(UA)$ has a certain thickness.

\begin{thm}\label{lem4}
Assume that $A$ is a self-adjoint operator satisfying that
$\sigma(A)=[\lambda_1, \lambda_2]$, $\lambda_1>0$ and
$\sigma_p(A)=\varnothing$. Then, there exists a unitary spread
operator $U$ such that $\sigma_p(UA)=R$ for any rings $R$ in the
ring $R_{\lambda_1, \lambda_2}^o=\{\lambda,
|\lambda_1|<|\lambda|<\lambda_2\}$.
\end{thm}

\begin{proof}

There is a sequence $\alpha_n\longrightarrow\lambda_2$ such that
$\alpha_{n+1}
>\alpha_{n}$ for each $n\geq 1$. Moreover, the range of spectral projection
$E_{[\alpha_n,\alpha_{n+1}]}$ is an infinite subspace;  and a
sequence $\beta_n\longrightarrow\lambda_1$ such that $\beta_{n}
>\beta_{n+1}$ for each $n\geq 1$. Moreover, the range of spectral projection
$E_{[\beta_{n+1},\beta_{n}]}$ is an infinite subspace.

Now we rearrange these intervals as follows.

$$J_n=[\alpha_n,\alpha_{n+1}), n\geq 0,$$
$$J_n=[\beta_{-n+1},\beta_{-n}), n\leq -1.$$

Let $I_0=J_0, I_n=J_{-n+1}$ for $n\geq 1$, $I_{-1}=J_{-1},
I_n=J_{-n}$ for $n\leq -1$.

Denote $E_n=E_{I_n}$ the spectral projection on the interval $I_n$
and by $H_n=$Ran$(E_n)$ for $n\in \mathbb{Z}$. Now we choose an ONB
$\{e_k^{(n)}\}_{k=1}^\infty$, for each $n\in \mathbb{Z}$. Since each
$H_n$ is a reducing subspace of $A$, we can write $A$ into the
direct sum:
$$A=\oplus_{n=-\infty}^{+\infty}A_n.$$
And $\alpha_0||x||\leq||A_0x||\leq\alpha_1$ for $x\in H_{0}$,
$\beta_{0}|x||\leq||A_{-1}x||\leq\beta_{-1}||x||$ for $x\in H_{-1}$,
$\beta_{-n}||x||\leq||A_nx||\leq\beta_{-n-1}$ for $x\in H_{n}$,
$n\geq 1$, $\alpha_n||x||\leq||A_0x||\leq\alpha_{n+1}$ for $x\in
H_{n}$ $n\leq -1$.

Now let $U$ be the unitary spread operator defined as

$$Ue_k^{(n)}=e_k^{(n+1)}, n\in\mathbb{Z}, k\in \mathbb{N}.$$

For a vector $x\in \mathcal{H}$ now under the ONB constructed it has
a $l_2$-sequence coordinate in the form

$$x=\sum_{n\in\mathbb{Z}}x^{(n)}=\sum_{n\in\mathbb{Z}}\sum_{k=1}^{\infty}x_k^{(n)}e_k^{(n)}.$$

Now simply we have

$$UAx=UA(\sum_{n\in\mathbb{Z}}x^{(n)})=\sum_{n\in\mathbb{Z}}UAx^{(n)}=\sum_{n\in\mathbb{Z}}A_{n-1}x^{(n-1)}.$$

Now suppose for some $\lambda\neq 0$ we do have some vector $x$ such
that $(\lambda I-UA)x=0$, then we have

$$\lambda x^{(n)}=A_{n-1}x^{(n-1)}.$$

Therefore, following equations hold:

$$x^{(n)}=\lambda^{-n}A_{n-1}A_{n-2}\cdots A_0 x^{(0)}, n\geq 1,$$

$$x^{(n)}=\lambda^{-n}A_{n}^{-1}A_{n+1}^{-1}\cdots A_{-1}^{-1}x^{(0)}, n\leq -1.$$

Since $\alpha_0||x||\leq||A_0x||\leq\alpha_1$ for $x\in H_{0}$,
$\beta_{0}|x||\leq||A_{-1}x||\leq\beta_{-1}||x||$ for $x\in H_{-1}$,
$\beta_{-n}||x||\leq||A_nx||\leq\beta_{-n-1}$ for $x\in H_{n}$,
$n\geq 1$, $\alpha_n||x||\leq||A_0x||\leq\alpha_{n+1}$ for $x\in
H_{n}$ $n\leq -1$ and $\beta_n\longrightarrow\lambda_1,
\alpha_n\longrightarrow\lambda_2$, it is easy to see that if
$\lambda_1<|\lambda|<\lambda_2$ then
$\sum_{n\in\mathbb{Z}}||x^{(n)}||^2<\infty$; if
$\lambda_1\leq|\lambda|$ there exists $N$ such that $||x^{(n)}||>1$
for $n>N$, if $|\lambda|\geq\lambda_2$ there exists $N$ such that
$||x^{(-n)}||>1$ for $n>N$, i.e. if $\lambda_1<|\lambda|$ or
$|\lambda|>\lambda_2$ then
$\sum_{n\in\mathbb{Z}}||x^{(n)}||^2=\infty$. Hence,
$\sigma_p(UA)=\{\lambda, |\lambda_1|<|\lambda|<\lambda_2\}$.

The proof of the more general situation is similar to the
Proposition \ref{lem2}.
\end{proof}

\begin{rem}
By the Theorem \ref{lem4}, we know that if $A$ is a self-adjoint
operator satisfying that $\sigma(A)=[\lambda_1, \lambda_2]$,
$\lambda_1>0$ and $\sigma_p(A)=\varnothing$, then for any ring $R$
in $R_{\lambda_1, \lambda_2}^o$, there exists $UA\in
\mathcal{O}_{u}(A)$ such that $\sigma_S(UA)$ contains $R$. i.e.
$\sigma_S(UA)$  has a certain thickness, $\sigma_S(A)$ is thin. In
other words, there is no ONB such that $\lambda I-UA$ has a matrix
representation as a Schauder matrix for $\lambda\in R$.
\end{rem}

Trivial modifications adapt the proof of the Theorems of
\ref{prop1}, \ref{lem5} and \ref{lem4}, we can get the following
proposition:

\begin{prop}\label{thm1}
Assume that $A$ is a self-adjoint operator.

\textup{(i)} If $\sigma(A)\subseteq[\lambda_1, \lambda_2],
\lambda_1>0$ and $\lambda_1,\lambda_2\in \sigma_e(A)$, then there
exists a unitary spread operator $U$ such that $\sigma_p(UA)=R$ for
any rings $R$ in the ring $R_{\lambda_1, \lambda_2}^o$;

\textup{(ii)} If $\lambda_1,\lambda_2\in \sigma_e(A)$ and
$0<\lambda_1<\lambda_2$, then there exists a unitary spread operator
$U$ such that $\sigma_p(UA)\supseteq R$ for any rings $R$ in the
ring $R_{\lambda_1, \lambda_2}^o$. Moreover, if there exist sequence
$\{t_k\}$ and $\{r_k\}$ contained in $\sigma(A)$ and satisfy that
 $t_k > t_{k+1}, r_k < r_{k+1}$ for all $k$, then there
exists a unitary spread operator $U$ such that $\sigma_p(UA)=R$ for
any rings $R$ in the ring $R_{\lambda_1\lambda_2}=\{\lambda,
|\lambda_1|<|\lambda|<\lambda_2\}$; for any unitary operator $U$,
$\sigma_p(UA)\subset \{\lambda,
|\lambda_1|\leq|\lambda|\leq\lambda_2\}$ and
 $\textup{Card}\{\sigma_p(UA)\cap
R_{\lambda_i}\}\leq\textup{dim Ker}(\lambda_iI-A), i=1,2$;

\textup{(iii)} If there exists only one point $\lambda_1\in
\sigma_e(A)$,  $\{t_k\}$ and $\{r_k\}$ contained in $\sigma(A)$
 and satisfy that $t_k <t_{k+1}, r_k > r_{k+1}$, $t_k\rightarrow
\lambda_1, r_k\rightarrow \lambda_2$, and
$\sum_{n=1}^\infty\prod_{k=1}^n(\frac{t_k}{\lambda_1})^2<\infty$,
$\sum_{n=1}^\infty\prod_{k=1}^n(\frac{\lambda_1}{r_k})^2<\infty$.
Then there exists a unitary spread operator $U$ such that
$\sigma_p(UA)=\{\lambda, |\lambda|=\lambda_1\}$.
\end{prop}

%
%
%
%
%
%
%
%

%
%

As we know,
$\sigma(T)\supset\sigma_S(T)=\sigma_p(T)\cup\{\lambda\in\mathbb{C},
\overline{\textup{Ran}(\lambda I-T)}\neq\mathcal{H}\}$ for every
$T\in B(\mathcal{H})$. Hence, by the Proposition \ref{thm1}, we
obtain the main theorem:

\begin{thm}\label{thm2}
Assume that $A$ is a self-adjoint Schauder operator.

\textup{(i)} If $\sigma(A)\subseteq[\lambda_1, \lambda_2],
\lambda_1>0$ and $\lambda_1,\lambda_2\in \sigma_e(A)$, then there
exists a unitary spread operator $U$ such that the Schauder spectrum
 $\sigma_S(UA)\supseteq R$ for any rings $R$ in the ring $R_{\lambda_1,
\lambda_2}^o$;

\textup{(ii)} If $\lambda_1,\lambda_2\in \sigma_e(A)$ and
$0<\lambda_1<\lambda_2$, then there exists a unitary spread operator
$U$ such that the Schauder spectrum $\sigma_S(UA)\supseteq R$ for
any rings $R$ in the ring $R_{\lambda_1, \lambda_2}^o$;

\textup{(iii)} If there exists only one point $\lambda_1\in
\sigma_e(A)$,  $\{t_k\}$ and $\{r_k\}$ contained in $\sigma(A)$
 and satisfy that $t_k <t_{k+1}, r_k > r_{k+1}$, $t_k\rightarrow
\lambda_1, r_k\rightarrow \lambda_1$, and
$\sum_{n=1}^\infty\prod_{k=1}^n(\frac{t_k}{\lambda_1})^2<\infty$,
$\sum_{n=1}^\infty\prod_{k=1}^n(\frac{\lambda_1}{r_k})^2<\infty$.
Then there exists a unitary spread operator $U$ such that  the
Schauder spectrum $\sigma_S(UA)\supseteq\{\lambda,
|\lambda|=\lambda_1\}$.

\end{thm}

According to the Proposition \ref{thm1} and Theorem \ref{thm2}, we
know that if a self-adjoint operator $A$ has more than one points in
its essential spectrum, then there exists a unitary spread operator
$U$ such that $\sigma_p(UA)$ contains a ring which is depended by
the essential spectrum, i.e. there exists $UA\in \mathcal{O}_{u}(A)$
such that $\sigma_S(A)$ is thin and $\sigma_S(UA)$ has a certain
thickness; if there is only one point in the essential spectrum and
satisfies some conditions, then there exists a unitary spread
operator $U$ such that $\sigma_p(UA)$ contains the circumference
which is depended by the essential spectrum, i.e. there exists
$UA\in \mathcal{O}_{u}(A)$ such that $\sigma_S(A)$ is at most
denumerable and $\sigma_S(UA)$ is uncountable. Furthermore, by Lemma
\ref{propbj}, we know that if $\sigma_e(A)$ has only one point
$\lambda_1$ and $\{t_k\}$ (or $\{r_k\}$) contained in $\sigma(A)$
and satisfy that $t_k <t_{k+1}$ (or $r_k > r_{k+1}$),
$t_k\rightarrow \lambda_1$(or $r_k\rightarrow \lambda_1$), then for
any unity operator $U$, $\sigma_p(UA)\neq R_{\lambda_1}$.
 However, we don't know if there exist  $\{t_k\}$ and $\{r_k\}$ contained in $\sigma(A)$
and satisfy that $t_k <t_{k+1}, r_k > r_{k+1}$, $t_k\rightarrow
\lambda_1, r_k\rightarrow \lambda_2$, does there exist a unitary
operator $U$ such that $\sigma_p(UA)=\{\lambda,
|\lambda|=\lambda_1\}$. It is easy to know that if $A=\lambda I$,
then the point spectrum of $UA$ is at most denumerable for any
unitary operator. We call a normal operator $A$ is {\it
non-trivial}, if $A\neq \lambda I$ for any $\lambda\in\mathbb{C}$.
Hence, we have the following question:

\begin{ques}
Assume that $A$ is a non-trivial invertible self-adjoint operator,
and there exists only one point $\lambda_1\in \sigma_e(A)$,
 $\{t_k\}$ and $\{r_k\}$ contained in $\sigma(A)$
and satisfy that $t_k <t_{k+1}, r_k > r_{k+1}$, $t_k\rightarrow
\lambda_1, r_k\rightarrow \lambda_2$. Whether there must be a unity
operator $U$ such that $\sigma_p(UA)=\{\lambda,
|\lambda|=\lambda_1\}$?

\end{ques}

\nocite{liyk2,liyk/kua1}

\end{document}